\documentclass[a4paper,11pt]{amsart}

\usepackage[english]{my-shortcuts}
\usepackage{srcltx,algorithm,algorithmic}
\usepackage{tikz}

\begin{document}

\title[Invertibility of random submatrices and incoherence]{Invertibility of random submatrices via tail-decoupling and a Matrix Chernoff Inequality}
\author{St\'ephane Chr\'etien {\tiny and} S\'ebastien Darses}

\address{Laboratoire de Math\'ematiques, UMR 6623\\ 
Universit\'e de Franche-Comt\'e, 16 route de Gray,\\
25030 Besancon, France} 
\email{stephane.chretien@univ-fcomte.fr}

\address{LATP, UMR 6632\\
Universit\'e de Provence, Technop\^ole Ch\^{a}teau-Gombert\\
39 rue Joliot Curie\\ 13453 Marseille Cedex 13, France}
\email{darses@cmi.univ-mrs.fr}

\maketitle


\begin{abstract}
Let $X$ be a $n\times p$ real matrix with coherence $\mu(X)=\max_{j\neq j'} |X_j^tX_{j'}|$.
We present a simplified and improved study of the quasi-isometry property for most submatrices of $X$ obtained
by uniform column sampling. Our results depend on $\mu(X)$, the operator norm $\|X\|$ and the dimensions with
explicit constants, which improve the previously known values by a large factor. 
The analysis relies on a tail-decoupling argument, of independent interest, and a recent version of the Non-Commutative Chernoff inequality (NCCI). 
\end{abstract}


\section{Introduction}

\subsection{Problem statement}

Let $\R^{n\times p}$ denote the set of all $n\times p$ real matrices. For any $M\in\mathbb R^{n\times p}$, we denote by $M^{t}$ its transpose and by $\|\cdot\|$ its operator norm:
\bean
\|M\|&:=&\max_{x\in\R^p, \|x\|_2=1} \|Mx\|_2, \quad \|x\|_2^2=x^t x. 
\eean
Let $X\in\R^{n\times p}$ and $T$ be a random index subset of size $s$ of $\{1,\ldots,p\}$ drawn from the
uniform distribution. Let $X_T$ denote the submatrix obtained by extracting the columns $X_j$'s of $X$ indexed by $j\in T$. We say that 
$X_T$ is an $r_0$-quasi-isometry if $\|X_T^tX_T-\Id\|\le r_0$ (quasi-isometry property).
The goal of this paper is to propose a new upper bound for the probability that
the submatrix $X_T$ fails to be an $r_0$-quasi-isometry. In the sequel, we assume that the 
columns of $X$ have unit norm. 

Proving that the quasi-isometry property holds with high probability has applications in Compressed Sensing and high-dimensional statistics based on sparsity. The uniform version of the quasi-isometry property,
i.e., satisfied for all possible $T$'s, is called the Restricted Isometry Property (RIP) and has been widely studied for independent, identically distributed (i.i.d.) sub-Gaussian matrices \cite{Mendelson:ConstApp08}. Recent works 
such as \cite{CandesPlan:AnnStat09} proved that the quasi-isometry property holds with high probability 
for matrices with sufficiently small coherence $\mu(X):=\max_{j\neq j'} |X_j^tX_{j'}|$. Unlike checking the RIP, computing $\mu(X)$ can be achieved in polynomial time. Such types of result are therefore of great potential 
interest for a wide class of problems involving high-dimensional linear or nonlinear regression models. 

Let $\{\delta_j \}$ denote a sequence of i.i.d. Bernoulli 0--1 random variables with expectation $\delta$. Let $R$ denote the square diagonal "selector matrix" whose $j^{th}$ diagonal entry is $\delta_j$.
Following the landmark papers of Bourgain and Tzafriri \cite{Bourgain:IsrJM87} (see also \cite{ProbBSp91}) 
and Rudelson \cite{Rudelson:JFA99}, Tropp \cite{Tropp:CRAS08} established, in particular, a bound for $(\pE \|R(X^tX-\Id)R\|^\rho )^{1/\rho}$, $\rho\in [2,\infty)$. As in \cite{Rudelson:JACM07}, the proof heavily relies on the Non-Commutative Khintchine inequality.
Using Tropp's result, Cand\`es and Plan proved in \cite[Theorem 3.2]{CandesPlan:AnnStat09} that 
$X_T$ is a $1/2$-quasi-isometry with probability greater than $1-p^{-2\log(2)}$ when $s\le  p/(4\|X\|^2)$ and the coherence $\mu(X)$ is sufficiently small. The quasi-isometry property for $r_0=\frac12$ then holds with high probability under 
easily-checked assumptions on $X$.

\subsection{Our contribution} The present paper aims at giving a more precise and self-contained version of 
Theorem 3.2 in \cite{CandesPlan:AnnStat09}.
Our result yields explicit constants, which improve the previously known values by a large factor. The analysis relies on a tail-decoupling argument, of independent interest, and a recent version of a Non-Commutative Chernoff inequality (NCCI) \cite{Tropp:ArXiv10}.

\subsection{Additional notations} 
\label{not}
For $S\subset \left\{1,\cdots,p\right\}$, we denote by $|S|$ the cardinality of $S$.
Given a vector $x\in \R^{p}$, we set $x_{T}=(x_{j})_{j\in T}\in\R^{|T|}$.

We denote by $\|M\| _{ 1\rightarrow 2}$ the
maximum $l_2$-norm of a column of $M\in\R^{n\times p}$ and $\|M\|_{\max}$ is the maximum absolute entry of $M$.

In the present paper, we consider the 'hollow Gram' matrix $H$:
\bea
H &=& X^tX-\Id \label{holo}.
\eea
In the sequel, $R^\prime$ will always denote an independent copy of the selector matrix $R$. 
Let $R_s$ be a diagonal matrix whose diagonal is a random vector $\delta^{(s)}$ of length $p$, uniformly distributed on the set of all vectors with $s$ components equal to 1 and $p-s$ components equal to 0. Notice that when $\delta=s/p$, the support of the diagonal of $R$ has cardinality close to $s$ with high probability, by a standard 
concentration argument.

\section{Preliminary results}

\subsection{On Rademacher chaos of order $2$ }

Let $\{\eta_i\}$ be a sequence of i.i.d. Rademacher random variables. 
Theorem 3.2.2 in \cite[p.113]{Gine:Decoupling99} gives the following general result:
a Banach-valued homogeneous chaos $X$ of order $d$ 
\bean
 X &=& \sum_{1\le i_1<\cdots<i_d\le p}X_{ i_1\cdots i_d}\eta_{i_1}\cdots \eta_{i_d}
\eean
verifies
$\left(\pE \|X\|^q \right)^{\frac{1}{q}} \le  \left(\frac{q-1}{p-1}\right)^{d/2} \left(\pE \|X\|^p \right)^{\frac{1}{p}}$, $1<p<q<\infty$.

We give an elementary proof in the real case with $d=2$ and $q=2p=4$, which yields a better constant.

\begin{lemm} \label{chaos}
Let $x_{ij}\in\R$, $1\le i,j\le p$. The homogeneous Rademacher chaos of order $2$: 
$\xi =  \sum_{i<j} x_{ij} \eta_i \eta_j$
verifies
\bea
\pE \ \xi^4 & \le & 9\ \left(\pE \ \xi^2\right)^2.
\eea
\end{lemm}

\proof
The multinomial formula applied to $\xi$ raised to the positive power $q$, gives
\bea \label{multi}
\xi^q & = & \sum_{} \frac{q!}{\prod \alpha_{ij}!} \prod x_{ij}^{\alpha_{ij}} (\eta_i\eta_j)^{\alpha_{ij}},
\eea
where the sum is over all integers $\alpha_{ij}$'s, $i< j$, such that $\sum \alpha_{ij}=q$, and the products are over all the indices $(i,j)$, $i<j$, ordered via the lexicographical order, still denoted by '$<$'. As from now, let these conventions hold. 

Case $q=2$ --- The partitions of $2$ are $2+0's$ and $1+1+0's$. Consider the partition $1+1+0's$, say $\alpha_{kl}=\alpha_{k'l'}=1$ for some  $4$-uple $(k,l,k',l')$ with $k\le k'$. We have $(k,l)\neq(k',l')$, $k<l$ and $k'<l'$. Thus,
\bean
\pE [\eta_{k}\eta_{l}\eta_{k'}\eta_{l'}]= \left\{
\begin{array}{l}
\pE [\eta_{k}] \ \pE[\eta_{l}\eta_{k'}\eta_{l'}] \ (=0) \ \text{ if } k<k' \\
\\
\pE[\eta^2_{k}] \ \pE [\eta_{l}] \ \pE[\eta_{l'}]\  (=0) \ \text{ if } k=k'.
\end{array}
\right.
\eean 
Therefore, $\pE \ \xi^2$ only depends on the partition $2+0's$, and one has
\bea \label{xi2}
\pE \ \xi^2 & = & \sum_{i<j} x_{ij}^2.
\eea  

Case $q=4$ --- The partitions of $4$ are $4$, $2+2$, $3+1$, $2+1+1$ and $1+1+1+1$ (we now omit the zeros). 

First, using the same arguments as in the case $q=2$, we show that the terms in $\pE\ \xi^4$ corresponding to the partitions $3+1$ and $2+1+1$ vanish.

Second, the partitions $1+1+1+1$ involve four different couples $(i,i')$, $(j,j')$, $(k,k')$ and $(l,l')$ (recall that $i<i'$, etc., and that the couples are lexicographically ordered). The only terms corresponding to the partitions $1+1+1+1$ whose expectation does not vanish are of the form  
\bean
x_{i_1i_1'}x_{i_1i_2'}x_{i_2i_1'}x_{i_2i_2'}\ \eta_{i_1}^2 \eta_{i_1'}^2 \eta_{i_2}^2 \eta_{i_2'}^2=x_{i_1i_1'}x_{i_1i_2'}x_{i_2i_1'}x_{i_2i_2'},
\eean
i.e., the four couples $(i_1,i_1')<(i_1,i_2')<(i_2,i_1')<(i_2,i_2')$ are the vertices of a rectangle into the upper off diagonal part of the matrix $(x_{ij})$. We denote by $\cal R$ the set of all these rectangles whose vertices are lexicographically ordered.

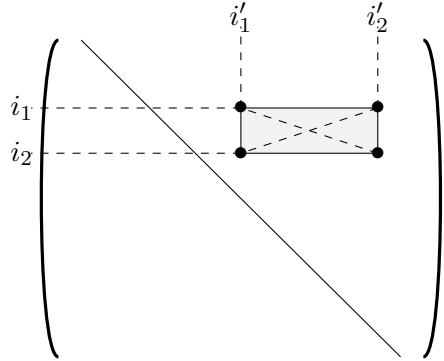
\begin{figure}[!h] 
\begin{center} \label{fig}
\begin{tikzpicture}[scale=.6]
\draw (1.5,7) -- (8.5,0);
\draw[fill=gray!10] (8,5.5) rectangle (5,4.5);
\draw[dashed] (8,5.5) -- (5,4.5);
\draw[dashed] (8,4.5) -- (5,5.5);

\draw[dashed] (8,5.5) -- (8,7.3);
\draw[dashed] (5,5.5) -- (5,7.3);

\draw[dashed] (5,4.5) -- (.4,4.5);
\draw[dashed] (5,5.5) -- (.4,5.5);

\draw (8,5.5) node {$\bullet$};
\draw (5,4.5) node {$\bullet$};
\draw (8,4.5) node {$\bullet$};
\draw (5,5.5) node {$\bullet$};

\draw[very thick] (1,0) .. controls (.5,0) and (.5,7) .. (1,7);
\draw[very thick] (9,0) .. controls (9.5,0) and (9.5,7) .. (9,7);

\draw (.2,5.5) node {$i_1$};
\draw (.2,4.5) node {$i_2$};

\draw (8,7.5) node {$i'_2$};
\draw (5,7.5) node {$i'_1$};
\end{tikzpicture}
\end{center}
\caption{The matrix $(x_{ij})$ where a 'rectangle' of $\cal R$ is drawn.}
\end{figure}

Finally, the $\alpha_{ij}$'s corresponding to the partitions $4$ and $2+2$ are even: $\alpha_{ij}=2\beta_{ij}$, with $\sum \beta_{ij}=2$. Therefore
\bean
\pE\ \xi^4 & = & \sum_{} \frac{4!}{\prod (2\beta_{ij})!} \prod x_{ij}^{2\beta_{ij}}+ \sum_{\cal R} 4!\ x_{i_1i_1'}x_{i_1i_2'}x_{i_2i_2'}x_{i_2i_1'} \ := A+B.
\eean
But 
\bean
A\le 3\sum_{} \frac{2!}{\prod \beta_{ij}!} \prod (x_{ij}^{2})^{\beta_{ij}}  =  3\left(\sum_{i<i'} x_{ii'}^2 \right)^2,
\eean 
and
\bean
B\le \frac{4!}{2}\sum_{\cal R} \left(x_{i_1i_1'}^2x_{i_2i_2'}^2 + x_{i_1i_2'}^2 x_{i_2i_1'}^2 \right)  \le  6\sum_{\stackrel{i<i', j<j'}{(i,i')<(j,j')}} 2!\  x_{ii'}^2x_{jj'}^2  = 6\left(\sum_{i<i'} x_{ii'}^2 \right)^2.
\eean
The second inequality for $B$ stems from relaxing the constraints induced by $\cal R$ and illustrated in Fig. \ref{fig}.
Using (\ref{xi2}), we obtain the desired result.
\qed

\begin{rema}
The ratio $\pE \ \xi^4 /(\pE \ \xi^2)^2$ will be used in the proof of Prop. \ref{condctrl}. We gain a factor $9$ compared to the constant $(\frac{4-1}{2-1})^{\frac22\cdot 4}=81$.
\end{rema}

\smallskip

\subsection{A Non-Commutative Chernoff inequality}

We will also need a corollary of a Matrix Chernoff's inequality recently established in \cite{Tropp:ArXiv10}.

\begin{theo}{\bf (Matrix Chernoff Inequality \cite{Tropp:ArXiv10})}
\label{bern}
Let $X_1$,\ldots,$X_p$ be independent random positive semi-definite matrices taking values in $\mathbb R^{d\times d}$. Set $S_p=\sum_{j=1}^p X_j$. Assume that for all $j\in\{1,\cdots,p\}$,
$\|X_j\|\le B$ a.s. and
\bean
 \left\| \pE\ S_p \right\| & \le & \mu_{\max}.
\eean
Then, for all $r\ge e\  \mu_{\max}$,
\bean
\bP\left( \left\| S_p \right\|\ge r \right) & \le & d \left(\frac{e\ \mu_{\max}}{r} \right)^{r/B}.
\eean
\end{theo}
(Set $r=(1+\delta)\mu_{\max}$ and use $e^\delta\le e^{1+\delta}$ in Theorem 1.1 \cite{Tropp:ArXiv10}.)

\section{Main results}

\subsection{Singular-value concentration theorem}

\begin{theo}\label{main1}
Let $r\in(0,1)$, $\alpha \ge 1$. Let us be given a full-rank matrix $X\in \mathbb R^{n\times p}$ and a
positive integer $s$, such that
\bea
\mu(X) & \le &  \frac{r}{2(1+\alpha)\log p}\\
s & \le & \frac{r^2}{4(1+\alpha) e^2}\ \frac{p}{ \|X\|^{2} \log p }.
\eea
Let $T\subset \left\{1,\ldots,p\right\}$ 
be a set with cardinality $s$, chosen randomly from the uniform distribution.
Then the following bound holds:
\bea \label{sing}
\bP \left(\|X_T^tX_T-\Id_s \|\ge r \right) & \le & \frac{216}{p^{\alpha}}.
\eea
\end{theo}


\subsection{Remarks on the various constants}\

The constant $216$ stems from the following decomposition:
$2$ (poissonization) $\times 36$ (decoupling)
$\times 3$ (union bound). This constant might look large. However, in many statistical applications  as in sparse models, $p$ is often assumed to be very large.

Let us now compare the constants $C_s$ and $C_\mu$ in the inequalities 
\bea
\mu(X) & \le &  \frac{C_\mu}{\log p} \label{form-cmu}\\
s & \le &C_s\ \frac{p}{ \|X\|^{2} \log p },
\eea
to the one of \cite{CandesPlan:AnnStat09}. The larger $C_s$ and $C_\mu$ are, the better the result is.

One of the various constraints on the rate $\alpha$ in \cite{CandesPlan:AnnStat09} is given by the theorem of Tropp in \cite{Tropp:CRAS08}. In this setting, $\alpha = 2\log 2$ and $r_{0} = 1/2$,
the author's choice of $1/2$ being unessential.
To obtain such a rate $\alpha$, they need to impose the r.h.s. of (3.15) in \cite{CandesPlan:AnnStat09} to be less than $1/4$, that is
$ 30 C_{\mu} + 13 \sqrt{2C_{s}}  \le  \frac{1}{4}$.
This yields $C_{s}  <  1.19 \times 10^{-4}$. Choosing $C_{s}$ close to $1.19 \times 10^{-4}$, e.g. 
$C_{s}  \simeq  1.18\ 10^{-4}$, we obtain:
\bean
C_{s}  \simeq  1.18\ 10^{-4},\quad
C_{\mu} \simeq  1.7\ 10^{-3}.
\eean

Our theorem allows to choose any rate $\alpha>0$. To make a fair comparison, let us choose $\alpha  = 2\log 2$ and $r  =  1/2$.
We obtain:
\bean
C_{s}  \simeq 3.5\ 10^{-3}, \quad
C_{\mu}  \simeq  0.1.
\eean

\section{Proof of Theorem \ref{main1}}
\label{Proofmain1}

In order to study the invertibility condition, we want to obtain bounds for the distribution tail of random sub-matrices of $H=X^tX-\Id$.

Let $R^{\prime}$ be an independent copy of $R$.
Let us recall two basic estimates:
\bean
\|H\|_{1\rightarrow 2}^2 \le  \|X\|^2,\quad
\|H\|^2 \le  \|X\|^4.
\eean

As a preliminary, let us notice that 
\beq \label{poisse}
\bP \left(\|R_sHR_s\|\ge r \right) \ \le \ 2\ \bP \left(\|RHR\|\ge r\right),
\eeq
which can be actually proven using the same kind of 'Poissonization argument' as in Claim $(3.29)$ p. 2173 in \cite{CandesPlan:AnnStat09}.

To study the tail-distribution of $\|RHR\|$, we use a decoupling technique which consists of replacing $\|RHR\|$ with $\|RHR'\|$.

\begin{prop}
\label{condctrl}
The operator norm of $RHR$ satisfies 
\bea \label{dec}
\bP\left(\|RHR\|\ge r\right) & \le & 36\ \bP\left(\|RHR^\prime\|\ge r/2\right).
\eea
\end{prop}
The main feature of this inequality is that the numerical constants are improved by a great factor when compared to the general result \cite[Theorem 1 p.224]{DeLaPena:BAMS94} (cf. Remark \ref{cste}). In addition to this decoupling argument, we need the following technical concentration result.

\begin{prop}\label{exp_bound}Let $X\in \mathbb R^{n\times p}$ be a full-rank matrix. For all parameters $s,r,u,v$ such that $\frac{p}{s} \frac{r^2}{e} \ge u^{2} \ge \frac{s}{p}\|X\|^{4}$ and 
$v^2 \ge \frac{s}{p}\|X\|^{2}$, the following bound holds:
\bea\label{inv_bound}
\bP \left(\|RHR'\|\ge r\right) & \le & 3 \ p\ \mathcal  V(s,[r,u,v]),
\eea
with
\bean
 \mathcal  V(s,[r,u,v]) & = & \left(e\frac{s}{p} \frac{u^2}{r^2} \right)^{\frac{r^2}{v^2}} +   \left(e \frac{s}{p}\frac{\|X\|^4}{u^2} \right)^{u^2/\|X\|^2} +  \left(e \frac{s}{p}\frac{\|X\|^2}{v^2} \right)^{v^2/\mu(X)^2}.
\eean
\end{prop}

\smallskip

We now have to analyze carefully the various quantities in Proposition \ref{exp_bound} in order to obtain for $P\left(\|RHR^\prime\|\ge r/2 \right)$ a bound of the order $e^{-\alpha \log p}$. 

Set $\alpha'=\alpha+1$ and $r'=r/2$. We tune the parameters so that
\bea
\frac{u^2}{\|X\|^2} & = & \alpha' \log p  \label{u} \\
\frac{v^2}{\mu(X)^2} & = & \alpha' \log p  \label{v} \\
\frac{r'^2}{v^2} & \ge & \alpha' \log p, \label{tv}
\eea
and 
\bea
e \frac{s}{p}\frac{\|X\|^4}{u^2} & \le & e^{-1} \label{u'} \\
e \frac{s}{p}\frac{\|X\|^2}{v^2} & \le & e^{-1} \label{v'} \\
e \frac{s}{p}\frac{u^2}{r'^2} & \le & e^{-1}. \label{ut}
\eea
A crucial quantity turns out to be $\frac{s}{p}\|X\|^{2}$. Keeping in mind that the hypothesis on the coherence reads
\bea \label{cmu}
\mu(X) & \le & \frac{C_{\mu}}{\log p},
\eea
it is necessary to impose that $s$ satisfies
\bea \label{sm}
\frac{s}{p} \|X\|^{2} & = & \frac{C_s}{\log p},
\eea
The constants $C_{\mu}$ and $C_{s}$ will be tuned according to several constraints.
The equalities (\ref{u}-\ref{v}) determine the values of $u$ and $v$. It remains to show that the previous inequalities are satisfied for a suitable choice of $C_{\mu}$ and $C_{s}$.

First, substituting (\ref{u}) into (\ref{ut}), we obtain:
\bean
\alpha' \frac{s}{p}\|X\|^2 \log p & \le e^{-2}r'^2.
\eean
Using (\ref{sm}), it follows that
\bea \label{csr}
C_s & \le & \frac{r'^2}{\alpha' e^2}.
\eea
Now, the bound (\ref{u'}) is satisfied if
\bean
\frac{e^2 C_s}{\log p} & \le & \alpha' \log p.
\eean
Based on (\ref{csr}), it suffices to have $\frac{r'^2}{\alpha'^2} \le \log^2 p$, that is
$p\ge e > e^{r'/\alpha'}$.

Second, substituting (\ref{v}) into (\ref{v'}), we obtain:
\bean
e^2 \frac{s}{p}\|X\|^2 & \le & \alpha' \mu(X)^2 \log p.
\eean
Using (\ref{cmu}) and (\ref{sm}), it follows that
\bean \label{cs}
e\sqrt{\frac{C_s}{\alpha'}} & \le & C_\mu.
\eean

Finally, (\ref{v}-\ref{tv}) yields
$r'^2 \ge  \alpha'^2 \mu(X)^2 \log^2 p$.
In view of (\ref{cmu}), it thus suffices to have $r' \ge  \alpha' \ C_\mu$.

To reach the desired conclusion, in order to ensure the six previous constraints, it suffices to choose $C_s$ and $C_\mu$ such that:
\bean
C_\mu  \le  \frac{r'}{1+\alpha} \quad {\rm and} \quad
C_s  \le \min\left( \frac{r'^2}{(1+\alpha) e^2}, (1+\alpha) \frac{C_\mu^2}{e^2} \right).
\eean
This completes the proof of Theorem \ref{main1}.

\section{Proof of the tail-decoupling and the concentration result}

\subsection{Proof of Proposition \ref{condctrl}}

Let us write
\bean
RHR & = & \sum_{i\neq j} \delta_{i}\delta_{j} H_{ij}.
\eean

Let $\{\eta_i\}$ be a sequence of i.i.d. independent Rademacher random variables, mutually independent of 
$\mathcal  D := \{\delta_i, 1\le i\le p\}$.
Following Bourgain and Tzafriri \cite{Bourgain:IsrJM87}, and de la Pe\~{n}a and Gin\'e \cite{Gine:Decoupling99}, we construct an auxiliary random variable:
\bean
Z=Z(\eta,\delta):=\sum_{i\neq j} (1-\eta_{i}\eta_{i}) \delta_{i}\delta_{j} H_{ij}.
\eean
Setting $Y = \sum_{i\neq j} \delta_{i}\delta_{j} H_{ij} \eta_i \eta_j$, we can write
\bea
Z & =& RHR + Y. \label{expand}
\eea

For the sake of completeness, we recall basic arguments from Corollary 3.3.8 p.12 in de la Pe\~{n}a and Gin\'e \cite{Gine:Decoupling99} (applied to (\ref{expand})) to obtain a lower bound for $\bP(\|Z\| \ge \|RHR\|)$. (We henceforth work conditionally on $\mathcal  D$.) 

Hahn-Banach's theorem gives a linear form $x^*$ on $\cal  \R^{p\times p}$ such that
\bea \label{x*}
\bP(\|Z\| \ge \|RHR\| \ |\cal D) & \ge & \bP(x^*(Z) \ge x^*(RHR) \ |\cal D) \nonumber\\
	 & \ge & \bP(x^*(Y) \ge 0\ |\cal D).
\eea
For any centered real random variable $\xi$, one obtains using H\"{o}lder's inequality twice (first with $\pE |\xi|=2\pE \ \xi \1_{\xi>0}$, second with $\pE\ \xi^2=\pE\ \xi^{2/3}\xi^{4/3}$):
\bea\label{xi}
 \bP(\xi \ge 0) \ge  \frac14 \frac{(\pE |\xi|)^2}{\pE\ \xi^2} \ge \frac14 \frac{(\pE \ \xi^2)^2}{\pE\ \xi^4}.
\eea 
Noticing that $x^*(Y)$ is a centered homogeneous real chaos of order $2$, we deduce from (\ref{x*}), (\ref{xi}) and Lemma \ref{chaos},
\bea
\bP \left( \left\| Z\right\| \ge \|RHR\| \  | \mathcal  D \right) &\ge & \frac{1}{4\times 9}=\frac{1}{36}.
\eea
Multiplying both sides by $\1_{\{ \|RHR\|\ge r \}}$ and taking the expectation, one has
\bea \label{324}
\frac{1}{36}\ \bP(\|RHR\|\ge r ) & \le &  \bP \left( \left\|  Z\right\| \ge r \right).
\eea

As from now, we can use similar arguments to \cite[Prop. 2.1]{Tropp:CRAS08}. There is a $\eta^{*}\in\{-1,1\}^{p}$ for which 
\bean
 \bP \left( \left\|  Z\right\| \ge r \right) = \pE\ \pE \left[ \1_{ \{\| Z\| \ge r\} } |(\eta_{i})\right] & \le & \pE\  \1_{ \{\| Z(\eta^*,\delta)\| \ge r\} }=\bP(\| Z(\eta^*,\delta)\| \ge r).
\eean
Hence, setting $T=\{i, \eta^{*}_{i}=1\}$, we can write
\bean
Z(\eta^*,\delta) & = &  2 \sum_{j\in T, \ k\in T^{c}} \delta_{j}\delta_{k} H_{jk} + 2 \sum_{j\in T^c, \ k\in T} \delta_{j}\delta_{k} H_{jk}.
\eean
Since $H$ is hermitian, we have
\bean
\left\| \sum_{j\in T, \ k\in T^{c}} \delta_{j}\delta_{k} H_{jk} + \sum_{j\in T^c, \ k\in T} \delta_{j}\delta_{k} H_{jk}\right\|  =  \left\| \sum_{j\in T, \ k\in T^{c}} \delta_{j}\delta_{k} H_{jk} \right\|.
\eean
Now, let $(\delta_i')$ be an independent copy of $(\delta_i)$. Set $\wdt \delta_i=\delta_i$ if $i\in T$ and $\wdt \delta_i=\delta'_i$ if $i\in T^c$. Since the vectors $(\delta_i)$ and $(\wdt \delta_i)$ have the same law, we then obtain:
\bean
\bP\left(\|Z\|\ge r\right) & \le & \bP\left( 2\ \left\| \sum_{j\in T, \ k\in T^{c}} \delta_{j}\delta'_{k} H_{jk}\right\|\ge r\right).
\eean
Re-introducing the missing entries in $H$ yields
\bean
\bP \left( \left\|  Z\right\| \ge r \right)  & \le &   \bP(\|RHR'\|\ge r/2 ),
\eean
which concludes the proof of the lemma due to (\ref{324}).


\begin{rem} \label{cste}
The previous result can be seen as a special case of Theorem 1 p.224 of the seminal paper \cite{DeLaPena:BAMS94}. Tracing the various constants involved in this theorem, we obtained the inequality
\bea \label{decbad}
\bP\left(\|RHR\|\ge r\right) & \le & 10^3\ \bP\left(\|RHR'\|\ge \frac{r}{18}\right).
\eea
\end{rem}


\subsection{Proof of Proposition \ref{exp_bound}}

We first apply the NCCI to $\|RHR'\|$ by conditioning on $R$.

\begin{lem}
The following bound holds:
\bea
P\left(\|RHR^\prime\|\ge r\right) & \le &  \bP\left(\|RH\| \ge u\right) + \bP\left( \|RH\|_{1\rightarrow 2} \ge v\right) \nonumber\\
& &  +\ p \left(e\frac{s}{p}\ \frac{u^2}{r^2} \right)^{\frac{r^2}{v^2}}.
\eea
\end{lem}

\begin{proof}

We have $\|RHR^\prime\|^{2}=\|RH R^{\prime 2} HR\|$. But ${R^\prime}^2=R^\prime$, so
\bea
\bP\left(\|RHR^\prime\|\ge r\right) & = & P\left(\|RH R^\prime HR\|\ge r^2\right).
\eea
We will first compute the conditional probability
\bea
\bP\left(\|RHR^\prime HR\|\ge r^2\mid R\right) & := & \pE\left[\1_{\{\|RHR^\prime HR\|\ge r^2\}}|\ R\right].
\eea
Let $Z_j$ be the $j^{th}$ column of $RH$, $j\in\{1,\cdots,p\}$. Notice that
\bean
RHR^\prime HR \ = \ \sum_{j=1}^p \delta_j^\prime Z_jZ_j^t \ := \  \sum_{j=1}^p A_j.
\eean

Since $\sum_{j=1}^p Z_jZ_j^t =RH^2R$ and $\|Z_jZ_j^t\|=\|Z_j\|_2^2$ ,
we then obtain
\bea
\|A_j\| & \le & \|RH\|_{1\rightarrow 2}^2\\
\left\| \sum_{j=1}^p \pE A_j \right\|& \le & \frac{s}{p}  \|RH\|^2.
\eea
The NCCI then yields
\bea
\bP\left( \|RHR^\prime HR\|\ge r^2 \mid R \right) &  \le & \ p  \left(e \frac{s}{p}\frac{\|RH\|^2}{r^2} \right)^{r^2/\|RH\|_{1\rightarrow 2}^2},
\eea
provided that
\bea
e \frac{s}{p}\frac{\|RH\|^2}{r^2}  & \le & 1.
\eea
Let us now introduce the events
\bean
\mathcal  A  =  \left\{ \|RHR^\prime HR\|\ge r^2 \right\};\quad 
\mathcal  B = \left\{\|RH\|\ge u \right\}; \quad
\mathcal  C = \left\{\|RH\|_{1\rightarrow 2} \ge v \right\}.
\eean
We have
\bean
\bP(\mathcal  A) & = & \bP(\mathcal  A\mid \mathcal  B\cup \mathcal  C)\bP(\mathcal  B\cup \mathcal  C)+\bP(\mathcal  A\cap \mathcal  B^c\cap \mathcal  C^c)\\
    & \le & \bP(\mathcal  B)+ \bP(\mathcal  C) +\bP(\mathcal  A\cap \mathcal  B^c\cap \mathcal  C^c).
\eean
The identity
$
\bP(\mathcal  A\cap \mathcal  B^c \cap \mathcal  C^c) = \pE \left[\1_{\mathcal  A\cap \mathcal  B^c \cap \mathcal  C^c} \right]
=\pE \left[\bP\left(\mathcal  A \mid R\right) \: \1_{\mathcal  B^c \cap \mathcal  C^c} \right]
$
concludes the lemma.
\hfill
\end{proof}

We now have to control the norm of $\frac{s}{p}RH^2R$, the norm of $RH$ and the column norm of $RH$.
Let us begin with $\|RH\|=\|HR\|$.

\begin{lem}
The following bounds hold:
\bean
P\left( \|HR\| >u \right) &\le & p  \left(e \frac{s}{p}\frac{\|X\|^4}{u^2} \right)^{u^2/\|X\|^2}\\
\bP\left( \|RH\|_{1\rightarrow 2}\ge v\right) &\le  & p  \left(e \frac{s}{p}\frac{\|X\|^2}{v^2} \right)^{v^2/\mu(X)^2},
\eean
provided that $e \frac{s}{p}\frac{\|X\|^4}{u^2}$ and $e \frac{s}{p}\frac{\|X\|^2}{v^2}$ are less than $1$.
\end{lem}

\begin{proof}

The steps are of course the same as what we have just done in the proof of Lemma \ref{condctrl}. Notice that
\bean
\bP\left( \|RH\| >u \right) =  \bP\left( \|HR\|^2 >u^2 \right)  =  \bP\left( \|HRH\| >u^2 \right).
\eean
The $j^{th}$ column of $H$ is $H_{j}=X^tX_j-e_j$. Moreover,
\bea
HRH=\sum_{j=1}^p \delta_j H_j H_j^t.
\eea
We have
$\ \|H_j H_j^t \| \ = \  \|H_j \|_2^2 \ \le\   \|H \|_{1\rightarrow 2}^2 \ \le \ \|X\|^2$,
and
\bea
\left\| \sum_{j=1}^p \pE [\delta_j H_j H_j^t] \right\|& \le & \frac{s}{p}  \|H\|^2 \  \le\   \frac{s}{p} \|X\|^4.
\eea
We finally deduce from the NCCI that
\bea
\bP\left(\|HRH\| \ge u^2 \right) &  \le & \ p  \left(e \frac{s}{p}\frac{\|X\|^4}{u^2} \right)^{u^2/\|X\|^2}.
\eea


Let us now control the supremum $\ell_2$-norm of the columns of $RH$. Set
\bea
M & = & \sum_{k=1}^{p} \delta_{k}\ {\rm diag}(H_{k}H^{t}_{k}).
\eea
Notice that
\bean 
\|RH\|^{2}_{1\rightarrow 2} & = \max_{k=1}^p \|(RH)_k\|_2^2= \left\|{\rm diag} \left((RH)^tRH \right) \right\| = & \left\|{\rm diag} \left(H^tRH \right) \right\|.
\eean
Thus,  
\bean
\|RH\|^{2}_{1\rightarrow 2} & = & \left\|{\rm diag} \left(\sum_{k=1}^p \delta_k (H^t)_kH_k^t \right) \right\|. 
\eean
Using symmetry of $H$ and interchanging the summation and the "diag" operation, 
we obtain that $\|RH\|^{2}_{1\rightarrow 2} = \|M\|$. Moreover, we have for all $k\in \{1,\cdots,p\}$,
\bea
\| {\rm diag}(H_{k}H^{t}_{k})\| & = \max_{j=1}^p (X_jX_k)^2  \le & \mu(X)^{2},
\eea
and
\bean
\| \pE M \|  = \frac{s}{p}  \| {\rm diag}(HH^{t})\|^2   = \frac{s}{p} \|H\|^{2}_{1\rightarrow 2} \le \frac{s}{p} \|X\|^2.
\eean
Applying the NCCI completes the lemma.
\end{proof}


\subsection*{Ackowledgment} The authors thank the referee for valuable comments that improved the paper. They thank Max H\"ugel for pointing out a mistake in a constant involved in a previous version of the tail-decoupling inequality.

\bibliographystyle{amsplain}
\bibliography{database}

\end{document}